\documentclass{amsart}

\usepackage[english]{babel}
\usepackage[utf8]{inputenc}

\usepackage{graphicx}
\usepackage{color}
\usepackage{caption}
\usepackage{subcaption}
\usepackage{stmaryrd}

\usepackage{hhline}

\usepackage{mathtools}
\usepackage{xcolor}

\usepackage{amsmath}
\usepackage{amssymb}
\usepackage{algorithm}

\usepackage{enumitem}
\usepackage{diagbox}

\usepackage{mathabx}
\usepackage{tikz}
\usetikzlibrary{graphs}
\usetikzlibrary{calc}

\usepackage{hyperref}
\usepackage{multido}

\usepackage{algpseudocode}

\numberwithin{equation}{section}

\DeclarePairedDelimiter\abs{\lvert}{\rvert}%
\newtheorem{theorem}{Theorem}[section]
\newtheorem{lemma}[theorem]{Lemma}

\newtheorem{assumption}{Assumption}

\theoremstyle{definition}

\calclayout

\DeclareMathOperator*{\argmin}{argmin \,}

\title{Recovering orthogonal tensors under arbitrarily strong, but locally correlated, noise}
 
 \author{Oscar Mickelin}
\address{Department of Mathematics, Massachusetts Institute of Technology, Massachusetts, USA}
\email{oscarmi@mit.edu}

\author{Sertac Karaman}
\address{Department of Aeronautics and Astronautics, Massachusetts Institute of Technology, Massachusetts, USA}
\email{sertac@mit.edu}



\begin{document}
\maketitle
\begin{abstract}
We consider the problem of recovering an orthogonally decomposable tensor with a subset of elements distorted by noise with arbitrarily large magnitude. We focus on the particular case where each mode in the decomposition is corrupted by noise vectors with components that are correlated locally, i.e., with nearby components. We show that this deterministic tensor completion problem has the unusual property that it can be solved in polynomial time if the rank of the tensor is sufficiently large. This is the polar opposite of the low-rank assumptions of typical low-rank tensor and matrix completion settings. We show that our problem can be solved through a system of coupled Sylvester-like equations and show how to accelerate their solution by an alternating solver. This enables recovery even with a substantial number of missing entries, for instance for $n$-dimensional tensors of rank $n$ with up to $40\%$ missing entries.
\end{abstract}

\section{Introduction}
We consider the problem of reconstructing a tensor $T$ in $\mathbb{R}^{n_1 \times \cdots \times n_d}$ of the form
\begin{equation}\label{eq:intro}
T = \sum_{k=1}^r v_{k1} \otimes v_{k2} \otimes \cdots \otimes v_{kd},
\end{equation}
from observations where its elements have been corrupted by noise. We will be interested in the case of orthogonally decomposable tensors $T$, i.e., one or more of the sets $\{v_{1j}, v_{2j} , \ldots , v_{rj} \}$ consists of orthogonal vectors, for some index $j$ in $1, \ldots , d$. These tensors arise for instance in independent component analysis \cite{comon1994independent} and learning latent variable models \cite{anandkumar2014tensor}.

The tractability of this problem strongly depends on the specific kind of noise considered. For general noise terms $E$, the problem becomes that of finding an optimal orthogonally decomposable approximation to the perturbed tensor $T + E$. This is in general NP-hard \cite{hillar2013most,mickelin2019optimal}, and several different approximation techniques have been proposed in the recent literature. Approaches using perturbations of tensor power iteration \cite{zhang2001rank,anandkumar2014tensor,mu2015successive,mu2017greedy} enjoy useful stability properties and can therefore cope with noise terms $E$ of bounded magnitude. Jacobi-type algorithms \cite{martin2008jacobi,li2019jacobi,usevich2019approximate} and algorithms based on polar decompositions \cite{chen2009tensor,hu2019linear}, alternating least-squares \cite{wang2015orthogonal} or singular value-decompositions \cite{guan2019numerical} have also seen increased interest in recent years. A different setting includes tensor completion problems, where the location of the non-zero elements of $E$ are typically chosen probabilistically. Reconstruction guarantees are then provided under the assumption that the tensor $T$ is of sufficiently low rank \cite{gandy2011tensor,liu2012tensor,goldfarb2014robust,kressner2014low,yuan2016tensor,gu2014robust,shah2015sparse,anandkumar2016tensor}. 

The goal of this article is to highlight how the structure of an orthogonally decomposable tensor enables reconstruction under arbitrarily strong, non-sparse noise, provided that the non-zero entries of the noise term are sufficiently structured, and that the rank of the tensor is above a certain threshold. Our results are meant to complement the tensor completion literature, since high-rank assumptions are required instead of the typical low-rank assumptions, and we achieve recovery for deterministic sampling patterns that cannot be efficiently tackled by e.g., nuclear-norm based techniques. Although simple, the techniques are surprisingly powerful, and are able to recover $n$-dimensional tensors of rank $n$ with up to $40\%$ missing entries. The outlook of the article is most closely related to the problem of decomposing a matrix into the sum of a matrix of low-rank and a diagonal matrix \cite{oseledets2006unifying,saunderson2012diagonal,saunderson2013diagonal}, but relies on different techniques. We would also like to mention recent articles on deterministic tensor completion \cite{ashraphijuo2017fundamental,ashraphijuo2019deterministic,ashraphijuo2020characterization,sorensen2019fiber}. The tensor structure in \eqref{eq:intro} encodes significant redundancy, which we will show enables a simple algebraic approach to exact recovery of the tensor $T$ and its factors $v_{kj}$. Focusing on a particular locally correlated noise model, we show how to construct coupled Sylvester-type equations that reconstruct $T$ even with a large number of unknown entries. We also show how an alternating linear solver then recovers $T$ efficiently. Our approach also easily extends to different sampling patterns.  

The remainder of this article is organized as follows. Section~\ref{sec:notation} introduces the notation, and section~\ref{sec:statement} the problem statement. Section~\ref{sec:alg_general} details the proposed algorithm and section~\ref{sec:numerical} shows numerical results. Implementations of the algorithms in this paper are publicly available online.\footnote{\href{https://github.com/oscarmickelin/locally-correlated-recovery}{\texttt{https://github.com/oscarmickelin/locally-correlated-recovery}}}

\section{Notation}\label{sec:notation}
For two integers $n$ and $m$ with $m \geq n$, we denote the range of integers contained (inclusively) between them by
\begin{equation}
\llbracket n, m \rrbracket := \{n, n+1, \ldots , m-1, m\}.
\end{equation}
We will denote the cardinality of a set of integers $\mathcal{I}$ by $\abs{\mathcal{I}}$.

The symbols $e_k$ for $k=1,\ldots , n$ will be reserved for the standard basis vectors in $\mathbb{R}^n$.

For vectors $v_1, \ldots , v_d$ with $v_k$ in $\mathbb{R}^{n_k}$, their tensor product $v_1 \otimes v_2 \otimes \cdots \otimes v_d$ is an element of $\mathbb{R}^{n_1 \times n_2 \times \cdots \times n_d}$ defined by
\begin{equation}
\left( v_1 \otimes v_2 \otimes \cdots \otimes v_d \right)(i_1, \ldots , i_d) := v_1(i_1)v_2(i_2)\cdots v_d(i_d).
\end{equation}

For two tensors $T$ and $S$ in $\mathbb{R}^{n_1\times \cdots \times n_d}$, we define their Euclidean inner product by
\begin{equation}
\langle T,S \rangle = \sum_{i_1 = 1}^{n_1} \cdots  \sum_{i_d = 1}^{n_d} T(i_1, \ldots , i_d) S(i_1, \ldots , i_d).
\end{equation}
The Euclidean norm of $T$ is then defined by $\|T\| = \sqrt{\langle T, T\rangle }$. 

For a tensor $T$ in $\mathbb{R}^{n_1\times \cdots \times n_d}$, and an index $\ell \in \{1,\ldots , n_d\}$, the $\ell$th slice of $T$ refers to the tensor $S$ in $\mathbb{R}^{n_1\times \cdots \times n_{d-1}}$ with elements
\begin{equation}
S(i_1, \ldots , i_{d-1}) = T(i_1, \ldots , i_{d-1}, \ell),
\end{equation}
for $i_k = 1, \ldots , n_k$, where $k=1, \ldots , d-1$.

For a matrix $M$ in $\mathbb{R}^{n\times m}$, denote the $\ell$th row of $M$ by $[M]_\ell$. We will also denote the set of non-zero elements of $M$ by
\begin{equation}
\text{supp}(M) = \Big\{(i,j) \in \llbracket 1,n\rrbracket \times \llbracket 1, m \rrbracket : M(i,j) \neq 0\Big\}.
\end{equation}
For a subset $\mathcal{I} \subseteq \llbracket 1, n \rrbracket  \times \llbracket 1, m \rrbracket$, we will denote the set of elements of $M$ contained in the index set $\mathcal{I}$ by
\begin{equation}
M \Big|_\mathcal{I} := \Big\{ M(i,j)  \in \mathbb{R} : (i,j) \in \mathcal{I} \Big\}.
\end{equation}
When accessing matrix elements, we will denote by a colon all the elements in the corresponding mode, e.g., $M(:,m)$ or $M(n,:)$.

\section{Problem statement}\label{sec:statement}
\subsection{Motivation}
As a motivation for the sampling patterns treated in this article, consider the problem of recovering a tensor $T = \sum_{k=1}^r a_k\otimes b_k \otimes c_k $ in $\mathbb{R}^{n\times n \times n}$ from a set corrupted samples of the form $\sum_{k=1}^r u_k(t) \otimes v_k(t) \otimes w_k(t) $, indexed by the variable $t$. Here
\begin{align}
u_k(t) &= a_k + n_{1,k}(t) ,\\
v_k(t) &= b_k + n_{2,k}(t) ,\\
w_k(t) &= c_k + n_{3,k}(t) ,
\end{align}
where $n_{j,k}(t)$ are noise terms with distributions symmetric around the origin. Our results allow for terms $n_{j,k}(t)$ with arbitrary magnitude, provided that their covariance terms obey a certain structure. To illustrate this statement, denote the covariance matrix of $n_{j,k}(t)$ and $n_{i,k}(t)$ by $Q_{ijk} := \mathbb{E}_t\left[n_{i,k}(t) n_{j,k}(t)^T\right]$. Since the distributions of the $n_{j,k}(t)$ are symmetric around the origin, the corrupted tensor can then be expanded in expectation as
\begin{equation}
\begin{split}\label{eq:motivaton_unknown}
&\mathbb{E}_t \! \left[ \sum_{k=1}^r \! u_k(t) \otimes v_k(t) \otimes w_k(t) \! \right] \!\! = \!\! \sum_{k=1}^r \! a_k\otimes b_k \otimes c_k + \mathbb{E}_t \! \left[ \sum_{k=1}^r  \! n_{1,k}(t) \otimes n_{2,k}(t) \otimes c_k \! \right] \\
& \qquad \qquad \quad \, + \mathbb{E}_t\left[ \sum_{k=1}^r  n_{1,k}(t) \otimes b_k \otimes n_{3,k}(t) \right] + \mathbb{E}_t\left[ \sum_{k=1}^r a_k \otimes n_{2,k}(t) \otimes n_{3,k}(t) \right] \\
&\qquad \qquad \quad \, = T + E.
\end{split}
\end{equation}
Here, we treat $E$ as a fully unknown tensor, with sparsity pattern determined by the sparsity patterns of the matrices $Q_{ijk}$. Each slice of the first term 
\begin{equation}
\mathbb{E}_t\left[ \sum_{k=1}^r  n_{1,k}(t) \otimes n_{2,k}(t) \otimes c_k \right] =\sum_{k=1}^r \mathbb{E}_t\left[  n_{1,k}(t) \otimes n_{2,k}(t)\right] \otimes c_k = \sum_{k=1}^r Q_{12k}\otimes c_k,
\end{equation}
has (in the general case) non-zero elements in the positions given by $\bigcup_{k=1}^r \text{supp}(Q_{12k})$. In other words, this first term $\mathbb{E}_t\left[ \sum_{k=1}^r  n_{1,k}(t) \otimes n_{2,k}(t) \otimes c_k \right]$ has support contained in
\begin{equation}
\left(\bigcup_{k=1}^r \text{supp}(Q_{12k}) \times \llbracket 1 , n\rrbracket\right).
\end{equation}

Similarly, the second two terms in \eqref{eq:motivaton_unknown} have support contained in permutations of the index-sets
\begin{equation}
\left(\bigcup_{k=1}^r \text{supp}(Q_{13k}) \times \llbracket 1 , n\rrbracket\right), \qquad \left(\bigcup_{k=1}^r \text{supp}(Q_{23k}) \times \llbracket 1 , n\rrbracket\right),
\end{equation}
respectively. Taken together, this specifies the pattern of the allowable non-zero elements of $E$. The problem of recovering the tensor $T$ can therefore be cast as a tensor completion problem with a deterministic sparsity pattern encoded by the support of $E$.

In particular, this article will focus on the case of locally correlated noise terms $n_{j,k}(t)$. By this, we mean terms for which the covariance terms $Q_{ijk}$ are band-diagonal, for all $i$, $j$, and $k$. In detail, we will specify a band-width $b$ and write the support of the covariance matrices as
\begin{equation}
\text{supp}(Q_{ijk}) = \Big\{ (i_1,i_2) \in \llbracket 1,n \rrbracket^2 :  \abs{i_1 - i_2} \leq b \Big\}.
\end{equation}
One can then verify that the resulting sparsity pattern for $E$ is of the form
\begin{equation}\label{eq:def_sparsity}
\Big\{ (i_1,i_2,i_3) \in \llbracket 1,n \rrbracket^3 :  \abs{i_k - i_\ell} \leq b , \text{ for some } k,\ell \in \{1,2,3\} \text{ with } k \neq \ell \Big\}.
\end{equation}

An illustration of this sparsity pattern is shown in figure~\ref{fig:fig_pattern}.

\begin{figure}
\centering
\includegraphics[width=\textwidth]{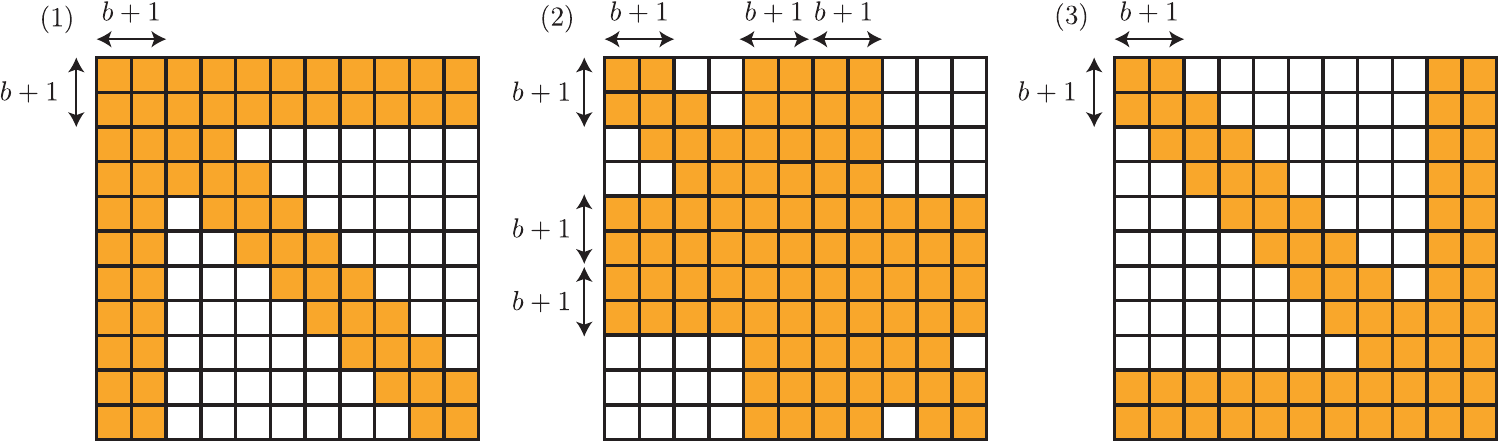}
\caption{Illustration of the sparsity pattern in \eqref{eq:def_sparsity}. Left: first slice of $E$. Middle: $\ell$th slice of $E$, for $b \leq \ell \leq n-b$. Right: $n$th slice of $E$.}
\label{fig:fig_pattern}
\end{figure}

However, we would like to emphasize that the techniques of this article are not by any means restricted to this specific sparsity pattern of $E$. Our approach enables recovery of $T$ for a variety of different sampling patterns encoded by $E$. In the three-dimensional case, two (non-exhaustive) examples are illustrated in figure~\ref{fig:fig_patterns}.

\begin{figure}
\centering
\includegraphics[width=\textwidth]{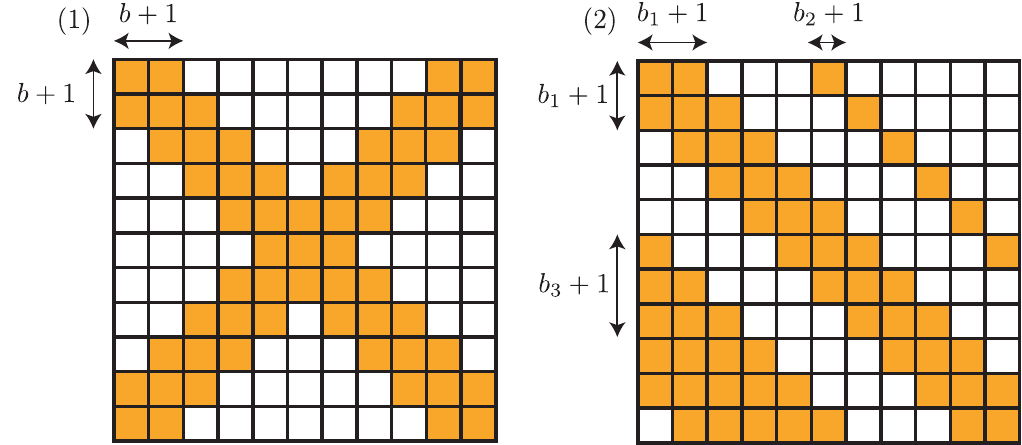}
\caption{Two illustrations of sparsity patterns of $E$ for which $T$ in \eqref{eq:motivaton_unknown} can be recovered from knowledge only of the corrupted tensor. Each slice of $E$ has the sparsity pattern indicated either in $(1)$ or $(2)$.}
\label{fig:fig_patterns}
\end{figure}

\subsection{Problem statement}
Consider an unknown tensor $T = \sum_{k=1}^r \otimes_{j=1}^d v_{kj}$ of rank $r$, which we aim to recover by knowledge of those elements of $T$ contained in a specific sampling pattern. We will consider a fixed sampling pattern, encoded by a tensor $E$. Although unknown, $E$ does, however, have a known sparsity pattern. We therefore seek to reconstruct all elements of $T$ from knowing only a corrupted version of $T$, denoted by $S$ and defined by
\begin{equation}\label{eq:defT}
S = T + E.
\end{equation}

We will in the following focus on the sparsity pattern in \eqref{eq:def_sparsity}, and in particular, study three-dimensional $n\times n \times n$ tensors. We therefore write $T = \sum_{k=1}^r a_k \otimes b_k \otimes c_k$.

Note that there clearly are examples of tensors $T$ that cannot be recovered from knowledge of $S$ alone. One set of examples consists of tensors $T$ with the same sparsity pattern as $E$. However, we will show that examples such as these are degenerate cases, in that almost all tensors $T$ can be recovered, under a few natural assumptions. Our results enable recovery under the assumption
\begin{assumption}\label{ass:ON}
The two sets of vectors $\{a_k\}_{k=1}^r$ and $\{b_k\}_{k=1}^r$ both consist of pairwise orthogonal vectors.
\end{assumption}

This assumption is reasonable to impose, since general tensor decomposition problems are NP-hard without similar assumptions \cite{hillar2013most}. Without loss of generality, we will assume that the vectors $a_k$ and $b_k$ are unit vectors, by absorbing their magnitudes into the vectors $c_k$.

Additionally, we will require a number of technical non-degeneracy conditions during the course of our recovery algorithm. They hold true for almost all collections of $\{a_k\}_{k=1}^r, \{b_k\}_{k=1}^r, \{c_k\}_{k=1}^r$, so for instance with probability one when these vectors are chosen at random. For ease of exposition, we will therefore state our results in terms of generic tensors. However, importantly, the degeneracy conditions can be explicitly verified during a run of the algorithm. Should they not hold in a specific use-case, this will be discovered and the user can remedy the deficiency by generating more data, which corresponds to reducing the bandwidth $b$. Our main result is the following.

\begin{theorem}\label{thm:main}
Assume that Assumption~\ref{ass:ON} holds, and let $T$ be a tensor of the form in \eqref{eq:defT} with unknown elements given by \eqref{eq:def_sparsity}. For generic matrices $A,B,C$, there are then constants $\alpha_1, \alpha_2, \beta_1, \beta_2 > 0$ such that $T$ can be uniquely recovered in polynomial time if 
\begin{equation}\label{eq:main_bound}
n \geq \alpha_1 b + \beta_1 ,
\end{equation}
and
\begin{equation}\label{eq:main_rank}
r \geq \alpha_2 b + \beta_2.
\end{equation}

An iterative recovery algorithm converges linearly to the true tensor, with cost $O\left( \text{max}( n^3, n^2r^2, n^2b^2N)\right)$, where $N$ is the number of required iterations.
\end{theorem}
The iterative recovery algorithm is presented in the next section. We would again like to emphasize that the rank bound in \eqref{eq:main_rank} enforces a high-rank assumption, which is principally the opposite of typical low-rank assumptions for tensor and matrix completion.

The proof of the theorem is constructive, and upper bounds on the constants $\alpha_1, \beta_1, \alpha_2, \beta_2$ can be read out from the proof. We also note lower bounds of the kind in theorem~\ref{thm:main} are required for our approach. In fact, we will show that tensors $T$ cannot be recovered with this approach without similar bounds, which leads to corresponding lower bounds. More precisely, the proof results in the bounds
\begin{equation}\label{eq:main_bounds}
\begin{split}
8(4b+2) &\geq \alpha_1 b + \beta_1 \geq \text{max} \! \left( \! 12b+7, \frac{93b - 69 + \sqrt{1419b^2 - 2704b + 921}}{10}  \right) ,\\
3(4b+2) &\geq \alpha_2 b + \beta_2 \geq 4b+1  .
\end{split}
\end{equation}

For large $b$, the lower bound on $n$ in \eqref{eq:main_bounds} becomes
\begin{equation}
n \geq \left(\frac{93}{10} + \frac{\sqrt{1419}}{10}\right)b - \frac{1352}{10\sqrt{1419}} - \frac{69}{10} + O\left(\frac{1}{b}\right) \approx 13.07b - 10.49 + O\left(\frac{1}{b}\right).
\end{equation}

However, the upper bounds in \eqref{eq:main_bounds} seem pessimistic empirically, and numerical experiments show that our techniques can in fact recover a sizable percentage of unknown elements, as shown in table~\ref{table:unknown_percentage}.

\begin{table}[ht]
\centering
\caption{The minimally admissible value of $n$ from \eqref{eq:main_bound} for different values of $b$, together with the percentage of unknown elements in the resulting tensor.}
\label{table:unknown_percentage}
\begin{tabular}{c|ccccc}
\hline\hline
$b$ & 1 & 3 & 5 & 7 &10   \\
\hline
$\text{Minimally admissible $n$}$ & 19 & 43& 68 & 94 &134  \\
Unknown elements ($\%$) & 40.5 & 41.4& 41.1& 40.7 & 40.0  \\
$\text{Minimally admissible $r$}$ &7& 15 & 23 &30 & 42
\end{tabular}
\end{table}

\section{Algorithm}\label{sec:alg_general}
\subsection{Overview}\label{sec:algorithm}
This section presents an algorithm to recover the tensor in theorem~\ref{thm:main}. The steps in the algorithm are as follows:
\begin{enumerate}
\item For a positive integer $m$, choose $m$ distinguished slices of $S$ (equation~\eqref{ex:choose_slices}). Using these slices, construct a set of $\binom{m}{2}$ coupled Sylvester-like equations (equation~\eqref{eq:solve_N}).
\item Solve the linear equations in the previous step, using an alternating least-squares procedure to accelerate convergence. This recovers parts of the $m$ chosen slices.
\item Construct a second set of coupled Sylvester-like equations (equation~\eqref{eq:solve_N2nd}) to recover the remaining entries of a subset of the $m$ distinguished slices.
\item Use the reconstructed slices to recover the vectors $a_k$ and $b_k$ for $k=1,\ldots , r$.
\item Recover the vectors $c_k$ by solving one linear system per slice.
\end{enumerate}

Sections~\ref{sec:details1}-\ref{sec:details5} below present the details of each step. We first introduce some notation. Write $A \in \mathbb{R}^{n\times r}$ for the matrix obtained by stacking the $r$ vectors $a_k$, $k = 1, \ldots , r$, side-by-side, and similarly for $B\in \mathbb{R}^{n\times r}$ and $C \in \mathbb{R}^{n\times r}$, i.e.,
\begin{align}
A &= \begin{bmatrix}
| & | & & | \\
a_1 & a_2 & \ldots & a_r \\
| & | & & |
\end{bmatrix} \! , \, \, \,
B = \begin{bmatrix}
| & | & & | \\
b_1 & b_2 & \ldots & b_r \\
| & | & & |
\end{bmatrix} \! , \, \, \,
C = \begin{bmatrix}
| & | & & | \\
c_1 & c_2 & \ldots & c_r \\
| & | & & |
\end{bmatrix}.
\end{align}

Define also a set of diagonal matrices $D^c_{\ell} \in \mathbb{R}^{r\times r}$ for $\ell = 1, \ldots , n$ by
\begin{equation}
D^c_{\ell} = \begin{bmatrix}
c_1(\ell) & \\
& c_2(\ell) & \\
& & \ddots & \\
& & & c_r(\ell)
\end{bmatrix} .
\end{equation}

We next define two index sets. Define $\mathcal{I}_k(b)$, for $k=1, \ldots , n$, by
\begin{equation}
\mathcal{I}_k(b) = \llbracket \text{max}(1, k - b), \text{min}(n,k+b) \rrbracket.
\end{equation}
Define also $\mathcal{D}(b)$ to be the indices contained in a band around the diagonal of width $b$ in the $x$- and $y$-directions, i.e., 
\begin{equation}
\mathcal{D}(b) := \Big\{(i,j) \in \llbracket 1,n\rrbracket^2 : \abs{i-j} \leq b \Big\}. 
\end{equation}
For each $k = 1, \ldots , n$, we let $\mathcal{U}(k)$ denote the set of matrices with support contained in the set of unknown elements of the $k$th slice of $T$. More precisely, let
\begin{equation}
\mathcal{U}(k) := \Big\{ M \in \mathbb{R}^{n\times n}: \text{supp}(M) \subseteq \mathcal{D}(b) \cup \Big( \mathcal{I}_k(b) \times \llbracket 1, n\rrbracket \Big) \cup \Big( \llbracket 1, n \rrbracket  \times \mathcal{I}_k(b) \Big) \Big\}.
\end{equation}

\subsection{Details of algorithm}
This section presents the details of the algorithm in section~\ref{sec:algorithm}.
\subsubsection{Step 1: recover parts of $m$ distinguished slices of $T$}\label{sec:details1}
We consider $m$ slices $k_1, \ldots , k_m$ of $S$, chosen so that the corresponding index sets $\mathcal{I}_{k_1}(b)$, $\mathcal{I}_{k_2}(b)$, $\ldots$, $\mathcal{I}_{k_m}(b)$ are disjoint. Explicitly, we take
\begin{equation}\label{ex:choose_slices}
k_i = (2b+1)(i-1)+1, \quad  \text{ for } i \in \llbracket 1, m \rrbracket .
\end{equation}
Note that this is possible when
\begin{equation}\label{eq:cond_disjoint}
(2b+1)(m-1)+1 \leq n.
\end{equation}

We will see below that recovery is guaranteed for any $m \geq 5$. Nonetheless, a higher value of $m$ provides more information to the algorithm and one therefore expects this to lead to possible recovery for larger values of $b$, compared to $n$ and $r$. However, the disjointness condition in \eqref{eq:cond_disjoint} introduces a bound on the highest possible value of $m$ in terms of $n$ and $b$, meaning that increasing $m$ only increases recovery performance up to a point.

From slices $k_1, \ldots , k_m$ of $S$, we have access to the matrices 
\begin{align} \label{eq:two_slices}
M_{i} &:= A D^c_{k_i}B^T + X_{i},  \quad \text{ for } i \in \llbracket 1,m \rrbracket ,
\end{align}
where $X_{i} \in \mathcal{U}(k_i)$, $i\in \llbracket 1,m \rrbracket$ are unknown matrices. The first step of the algorithm recovers parts of the $X_{i}$ from a set of Sylvester-type equations that we now construct.

For each pair $M_{i}, M_{j}$ for $i,j \in \llbracket 1, m\rrbracket, i \neq j$, note that
\begin{equation}\label{eq:solve_setup}
\begin{split}
(M_{i} - X_{i})\cdot (M_{j} - X_{j})^T - &(M_{j} - X_{j})\cdot (M_{i} - X_{i})^T \\
&= AD^c_{k_i}D^c_{k_j}A^T - AD^c_{k_j}D^c_{k_i}A^T = 0.
\end{split}
\end{equation}
Expanding the left hand side therefore results in the set of equations
\begin{equation}\label{eq:solve_N}
\begin{split}
X_{i}M_{j}^T - M_{j}X_{i}^T + M_{i}X_{j}^T &- X_{j}M_{i}^T  + X_{j}X_{i}^T - X_{i}X_{j}^T \\
&= M_{i}M_{j}^T - M_{j}M_{i} ^T, \quad i,j \in \llbracket 1, m\rrbracket.
\end{split}
\end{equation}
We will next restrict these sets of equations to the complement of the support of the quadratic terms $X_{j}X_{i}^T - X_{i}X_{j}^T$. The resulting linear equations will then be used to determine parts of the $X_{i}$ uniquely. We first write down the support of the quadratic terms, as follows.
\begin{lemma}
Define the index set 
\begin{equation}
\begin{split}
\mathcal{I}_{ij}(b) := \mathcal{D}(2b) & \cup \Big( \mathcal{I}_i(2b) \times \llbracket 1, n\rrbracket \Big) \cup \Big(  \llbracket 1, n\rrbracket \times \mathcal{I}_i(2b) \Big)  \\
& \cup \Big( \mathcal{I}_j(2b) \times \llbracket 1, n\rrbracket \Big) \cup \Big(  \llbracket 1, n\rrbracket \times \mathcal{I}_j(2b) \Big) .
\end{split}
\end{equation}
For any two $N_i \in \mathcal{U}(i)$, $N_j \in \mathcal{U}(j)$ with $\mathcal{I}_{i} (b) \cap \mathcal{I}_{j} (b) = \emptyset$, we have
\begin{equation}
\text{supp}(N_jN_i^T - N_iN_j^T ) \subseteq \mathcal{I}_{ij}(b).
\end{equation}
\end{lemma}
\begin{proof}
Write
\begin{align}
N_i = D_i +  \sum_{k \in \mathcal{I}_i(b)} \left( x_k \otimes e_k + e_k \otimes y_k \right) , \,\,\, N_j = D_j + \sum_{\ell \in \mathcal{I}_j(b)} \left( z_\ell \otimes e_\ell + e_\ell \otimes w_\ell \right),  
\end{align}
for vectors $x_k,y_k,z_\ell,w_\ell$ in $\mathbb{R}^n$ and where $D_i, D_j$ have support contained in $\mathcal{D}(b)$. We start with the term $N_iN_j^T$. In the expansion of the product $N_iN_j^T - N_jN_i^T$, the cross terms $D_iD_j^T$ and $D_jD_i^T$ have support contained in $\mathcal{D}(2b)$. The terms $D_i\left( z_\ell \otimes e_\ell + e_\ell \otimes w_\ell \right)$ have support in the rows in $\mathcal{I}_j(2b)$. Likewise, the terms $(x_k \otimes e_k) D_j^T$ have support in the columns with indices in $\llbracket \text{min}(\mathcal{I}_i(b)) - b , \text{max}(\mathcal{I}_i(b)) + b\rrbracket $, i.e., in $\mathcal{I}_i(2b)$. It follows that the terms $(x_k \otimes e_k) N_j^T$ have support in the columns with indices in $\mathcal{I}_i(2b) \cup \mathcal{I}_j(b)$, since $\mathcal{I}_i(b)$ and $\mathcal{I}_j(b)$ have empty intersection. Similarly, the terms $(e_k \otimes y_k) N_j^T$ have support contained in the rows with indices in $\mathcal{I}_i(b)$.

For the term $N_jN_i^T$, transposing the support of the term $N_iN_j^T$ and exchanging $i$ and $j$ concludes the proof.
\end{proof}

Restricting \eqref{eq:solve_N} to only the entries contained in the complement of $\mathcal{I}_{k_ik_j}(b)$ therefore results in a linear equation in the elements of $X_{i}$ and $X_{j}$. Note that only the rows of $X_{i}$ with indices not contained in $\mathcal{I}_{ k_i}(2b)$ appear in any of these equations. We therefore use the collection of these equations for all $i,j \in \llbracket 1, m\rrbracket$ to determine the rows of $X_{i}$ with indices not contained in $\mathcal{I}_{ k_i}(2b)$, as in the following lemma.

\begin{lemma}\label{lemma:count}
Let $X_i$ be matrices in $\mathcal{U}(k_i)$, for $i=1, \ldots , m$ and define $M_{i}$ in $\mathbb{R}^{n\times n}$ as in \eqref{eq:two_slices}. Assume $m \geq 5$.
\begin{enumerate}
\item For generically chosen matrices $A,B,C$ in $\mathbb{R}^{n\times r}$ satisfying assumption~\ref{ass:ON}, the equations
\begin{equation}\label{eq:sylvester}
\begin{split}
X_{i}M_{j}^T - M_{j}X_{i}^T &+ M_{i}X_{j}^T - X_{j}M_{i}^T\Big|_{\mathcal{I}_{k_ik_j}(b)^\complement}   \\
&= M_{i}M_{j}^T - M_{j}M_{i} ^T\Big|_{\mathcal{I}_{k_ik_j}(b)^\complement} , \quad i,j \in \llbracket 1, m\rrbracket , 
\end{split}
\end{equation}
uniquely determine the rows of $X_{i}$ with indices not contained in $\mathcal{I}_{k_i}(2b)$, provided
\begin{equation}\label{eq:cond_lemma}
\begin{split}
n &\geq 8(4b+2), \\
r &\geq 3(4b+2).
\end{split}
\end{equation}
\item If either
\begin{equation}\label{eq:cond_lemma_converse}
\begin{split}
n &\leq \text{max}\left(12b+7, \frac{93}{10}b + \frac{\sqrt{1419b^2 - 2704b + 921}}{10} - \frac{69}{10} \right) \quad \text{ or } \\
r &\leq 4b+1,
\end{split}
\end{equation}
then \eqref{eq:sylvester} has a non-trivial kernel when restricted to the rows of $X_{i}$ with indices not contained in $\mathcal{I}_{k_i}(2b)$.
\end{enumerate}
\end{lemma}

A proof is given in appendix~\ref{sec:proof_lemma_count}. Note that the second part of lemma~\ref{lemma:count} shows that high-rank assumptions of the type in \eqref{eq:cond_lemma} are indeed required in our approach.

\subsubsection{Step 2: solve the linear systems using alternating least-squares}\label{sec:details2}
The coupled linear systems in \eqref{eq:sylvester} constitute an overdetermined system of $O(n^2)$ equations in $O(nb)$ unknowns. A direct solution would therefore have complexity $O(n^4b^2)$. However, because of the particular structure of the system, convergence can be accelerated through an alternating minimization strategy. The least-squares solution of the linear system amounts to minimizing the Euclidean norm of the residual. The proof of lemma~\ref{lemma:count} will show that the associated linear system has full column rank, for generic choices of $A,B,C$, so this cost function is strongly convex. The minimizer of this problem can therefore be found by alternatingly minimizing the residual with respect to a subset of the variables \cite[Proposition 3.4]{luo1993error}. Iterating this procedure results in convergence to the solution of the linear system.

In detail, we alternate over each matrix $X_{i}$ for $i \in \llbracket 1, m\rrbracket $ in turn. Denote the $t$th iterate by $X_{i}^{(t)}$. For each $i$ and $t$, we in turn update each row, $[X_{i}^{(t)}]_\ell$, of $X_{i}^{(t)}$ for $\ell$ not contained in $\mathcal{I}_{k_i}(2b)$. Denote the result of updating rows $\llbracket 1, \ell - 1\rrbracket \smallsetminus \mathcal{I}_{k_i}(2b)$ of $X_{i}^{(t)}$ by $X_{i}^{(t,\ell-1)}$. The update equation of the $\ell$th row reads as
\begin{equation}\label{eq:update}
\begin{split}
[X_{i}^{(t+1)}]_\ell = & \argmin_{[X_{i}^{(t)}]_\ell} \sum_{\substack{ j=1 \\j \neq i}}^m  \Big\| X_{i}^{(t,\ell-1)}M_{j}^T - M_{j}X_{i}^{(t,\ell-1) T} \\
&+ M_{i}X_{j}^{(t,\ell-1)T} - X_{j}^{(t,\ell-1)}M_{i}^T - \left( M_{i}M_{j}^T - M_{j}M_{i} ^T\right)\Big|_{\mathcal{I}_{k_i k_j}{(b)}^\complement}  \Big\|^2 .
\end{split}
\end{equation}

The variables in $[X_{i}^{(t,\ell-1)}]_\ell$ only appear in the $\ell$th row and column of the matrix 
\begin{equation*}
X_{i}^{(t,\ell-1)}M_{j}^T - M_{j}X_{i}^{(t,\ell-1) T} + M_{i}X_{j}^{(t,\ell-1)T} - X_{j}^{(t,\ell-1)}M_{i}^T - \left( M_{i}M_{j}^T - M_{j}M_{i} ^T\right)\Big|_{\mathcal{I}_{k_i k_j}{(b)}^\complement}.
\end{equation*}
The update equation \eqref{eq:update} is therefore a least squares problem in $O(b)$ unknowns and $O(n)$ equations. It can be solved with cost $O(nb^2)$. The total complexity of the alternating procedure therefore becomes $O(n^2b^2N)$, where $N$ is the number of iterations used.

\subsubsection{Step 3: recover remaining entries of a subset of the distinguished slices of $T$} \label{sec:details3}
Next, we recover the remaining entries of $X_{i}$, for $i$ contained in a subset $I$ of $\llbracket 1, m \rrbracket$. We will require that $I$ satisfies
\begin{align}
&\text{$\abs{I} \geq 3$,} \label{eq:defI1}\\
&\text{For all $i,j \in I$ with $i \neq j$, $\mathcal{I}_{k_i}(2b) \cap \mathcal{I}_{k_j}(2b) = \emptyset$.} \label{eq:defI2}
\end{align}

Note the multiplicative factor $2$ in the second requirement. As an example, these requirements are necessarily satisfied for $ I = \{1,3,5\}$, if $m \geq 5$. A larger value of $\abs{I}$ will guarantee recovery under more beneficial bounds of $n$ and $r$ in terms of $b$, but we will phrase our results in terms of the limiting case $\abs{I} = 3$.

At this stage, the $i$th slice $M_{i}$ has all entries known except for elements in the rows with indices contained in $\mathcal{I}_{k_i}(2b)$. This is illustrated in figure~\ref{fig:fig_pattern_2nd}.

\begin{figure}
\centering
\includegraphics[width=\textwidth]{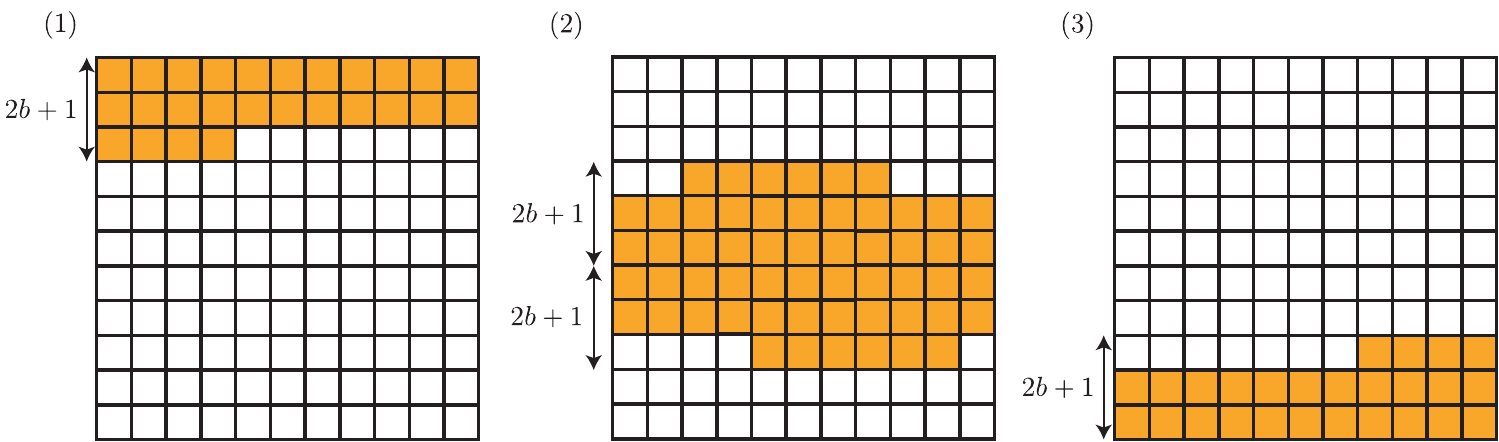}
\caption{Illustration of the remaining unknown elements in section~\ref{sec:details3}. Compare to figure~\ref{fig:fig_pattern}. Left: first slice of $E$. Middle: $\ell$th slice of $E$, for $b \leq \ell \leq n-b$. Right: $n$th slice of $E$.}
\label{fig:fig_pattern_2nd}
\end{figure}

We will recover these rows, for slices $k_i$ with $i$ contained in $I$. For this purpose, we construct the system of equations
\begin{equation}\label{eq:solve_N2nd}
(M_{i} - X_{i})^T\cdot (M_{j} - X_{j}) - (M_{j} - X_{j})^T\cdot (M_{i} - X_{i}) = 0, \quad i,j \in I,
\end{equation}
where each $X_{i}$ is contained in the subspace
\begin{equation}
\Big\{ X_i \in \mathcal{U}(k_i): \text{supp}(X_i) \subseteq  \mathcal{I}_{k_i}(2b) \times \llbracket 1, n\rrbracket  \Big\}.
\end{equation}

Note that the order of the terms within each product in \eqref{eq:solve_N2nd} is reversed, as compared to \eqref{eq:solve_setup}. Moreover, the quadratic cross terms in \eqref{eq:solve_N2nd} satisfy
\begin{align}
X_{i}^T X_{j}  = 0 = X_{j}^T X_{i}, \quad i,j \in  I , i \neq j,
\end{align}
since the sets $\mathcal{I}_{k_i}(2b)$ and $\mathcal{I}_{k_j}(2b)$ are disjoint by the second assumption on the set $I$. Equation~\eqref{eq:solve_N2nd} is therefore a set of coupled linear systems for the unknowns $X_{i}$. We will show that they determine the $X_{i}$ uniquely. 

We next present the main result of this section.
\begin{lemma}\label{lemma:unique2nd}
For generic $A$ and $B$ and $I$ satisfying \eqref{eq:defI1}--\eqref{eq:defI2}, equation~\eqref{eq:solve_N2nd} has a unique solution, provided that
\begin{equation}\label{eq:rank2nd}
\begin{split}
n &\geq 16b+4 \\
r &\geq 8b+4 
\end{split}
\end{equation}
\end{lemma}

The bound on $r$ in \eqref{eq:rank2nd} is slightly pessimistic and we observe a looser bound in the numerical experiments below. Equation~\eqref{eq:solve_N2nd} can be solved in an alternating fashion similar to \eqref{eq:update} and we omit the details of this step.

\subsubsection{Step 4: use the recovered slices to recover $a_k$, $b_k$}\label{sec:details4}
From the $\abs{I}$ distinguished slices recovered in the previous step, we can use Jennrich's algorithm \cite{harshman1970foundations} to recover the vectors $a_k$, $b_k$. In detail, write $\widehat{T} \in \mathbb{R}^{n\times n \times \abs{I}}$ as the tensor obtained by stacking the $\abs{I}$ recovered slices $ k_i, i \in I$ of $T$, i.e.,
\begin{equation}
\widehat{T}(i_1,i_2,i_3) = T(i_1,i_2,s_{i_3}),
\end{equation}
where the elements of $I$ are enumerated as $s_1, \ldots , s_{\abs{I}}$. Writing $\widehat{c}_\ell$ as the restriction of $c_\ell$ to the elements with indices in $I$, $\widehat{T}$ has the decomposition
\begin{equation}
\widehat{T} = \sum_{\ell=1}^r a_\ell \otimes b_\ell \otimes \widehat{c}_\ell,
\end{equation}
and the vectors $a_k$, $b_k$ for $k=1, \ldots , r$ can be recovered by an application of Jennrich's algorithm.

\subsubsection{Step 5: recover the $c_k$}\label{sec:details5}
Lastly, we use the $\ell$th slice of $S$, for $\ell = 1, \ldots , n$, to recover the elements $c_1(\ell), \ldots , c_r(\ell)$. From each slice, we have access to
\begin{equation}\label{eq:recover_C}
M_\ell = AD_\ell^c B^T + N_\ell,
\end{equation}
for $N_\ell$ an unknown matrix in $\mathcal{U}(\ell)$. Since $A$ and $B$ are known from the preceding step, restricting \eqref{eq:recover_C} to the complement of the support of $\mathcal{U}(\ell)$ results in a linear equation in $c_1(\ell), \ldots , c_r(\ell)$. We therefore construct the linear system
\begin{equation}\label{eq:system_last}
M_\ell(i,j) = \sum_{k=1}^r A(i,k)B(j,k) c_k(\ell), \quad (i,j) \in \text{supp}(\mathcal{U}(\ell))^\complement.
\end{equation}

\begin{lemma}\label{lemma:solve3d}
For generic $A$ and $B$, the system in \eqref{eq:system_last} has a unique solution, provided that $n \geq 8b + 4$.
\end{lemma}
\begin{proof}
Let $i_0 = \text{max}\left(\mathcal{I}_\ell (b) \right) + 1$, if $\ell   < n -b $, and $i_0 = \text{min}\left(\mathcal{I}_n (b) \right)- 1$ if $\ell \geq n-b $. Restrict the system \eqref{eq:system_last} to the indices $(i_0,j_1)$ for $(i_0,j_1) \in \text{supp}(\mathcal{U}(\ell))^\complement$ and $(j_2,i_0)$ where $(j_2,i_0) \in \text{supp}(\mathcal{U}(\ell))^\complement$. If we enumerate the possible values of $j_1$ as $s_1, \ldots , s_N$ and the values of $j_2$ as $t_1, \ldots , t_M$, the restricted system can be written in the form
\begin{equation}
\resizebox{0.9\linewidth}{!}{%
$
\begin{bmatrix}
A(i_0,1)B(s_1,1) &  A(i_0,2)B(s_1,2)  & \ldots & A(i_0,r)B(s_1,r) \\
\vdots & \vdots & \ddots & \vdots \\
A(i_0,1)B(s_N,1) &  A(i_0,2)B(s_N,2)  & \ldots & A(i_0,r)B(s_N,r) \\
A(t_1,1)B(i_0,1) &  A(t_1,2)B(i_0,2)  & \ldots & A(t_1,r)B(i_0,r) \\
\vdots & \\
A(t_M,1)B(i_0,1) &  A(t_M,2)B(i_0,2)  & \ldots & A(t_M,r)B(i_0,r)
\end{bmatrix}
\begin{bmatrix}
c_1(\ell) \\
c_2(\ell) \\
\vdots \\
c_r(\ell)
\end{bmatrix} = \begin{bmatrix}
M_\ell(i_0, s_1) \\
M_\ell(i_0, s_2) \\
\vdots \\
M_\ell(i_0, s_N) \\
M_\ell(t_1, i_0) \\
M_\ell(t_2, i_0) \\
\vdots \\
M_\ell(t_M,i_0)
\end{bmatrix}.
$
}
\end{equation}

Note that $i_0$ is not contained in either of the lists $s_1, \ldots , s_N$ or $t_1, \ldots , t_M$ so the matrix in the left hand side has full column rank for generic $A$ and $B$ if $M + N \geq r$. Since $M+N \geq 2(n - 2(2b+1))$ and $n \geq r$, the conclusion follows if we impose $2(n - 2(2b+1)) \geq n$. This is equivalent to $n \geq 8b + 4$, which concludes the proof.
\end{proof}
The system in \eqref{eq:system_last} therefore determines $c_1(\ell), \ldots , c_r(\ell)$ uniquely. Repeating this procedure for $\ell = 1, \ldots , n$ therefore recovers the vectors $c_1, \ldots , c_r$.

\section{Numerical results}\label{sec:numerical}
This section details different numerical results of the algorithm of the article. The tensors used in the simulations are all of the form $\sum_{k=1}^{r} a_k \otimes b_k \otimes c_k$, with $c_k$ having independently generated normal entries, and the $a_k,b_k$ randomly generated orthogonal vectors. All computations were carried out on a MacBook Pro with a 3.1 GHz Intel Core i5 processor and 16 GB of memory.

As stopping criterion for the iterative solutions of equations~\eqref{eq:update} and \eqref{eq:solve_N2nd}, we terminate the iterations once the relative improvement from one iteration to the next is below a threshold denoted by $\varepsilon_{\text{tol}}$. 

\begin{figure}[!h]
\centering
\includegraphics[width=\textwidth]{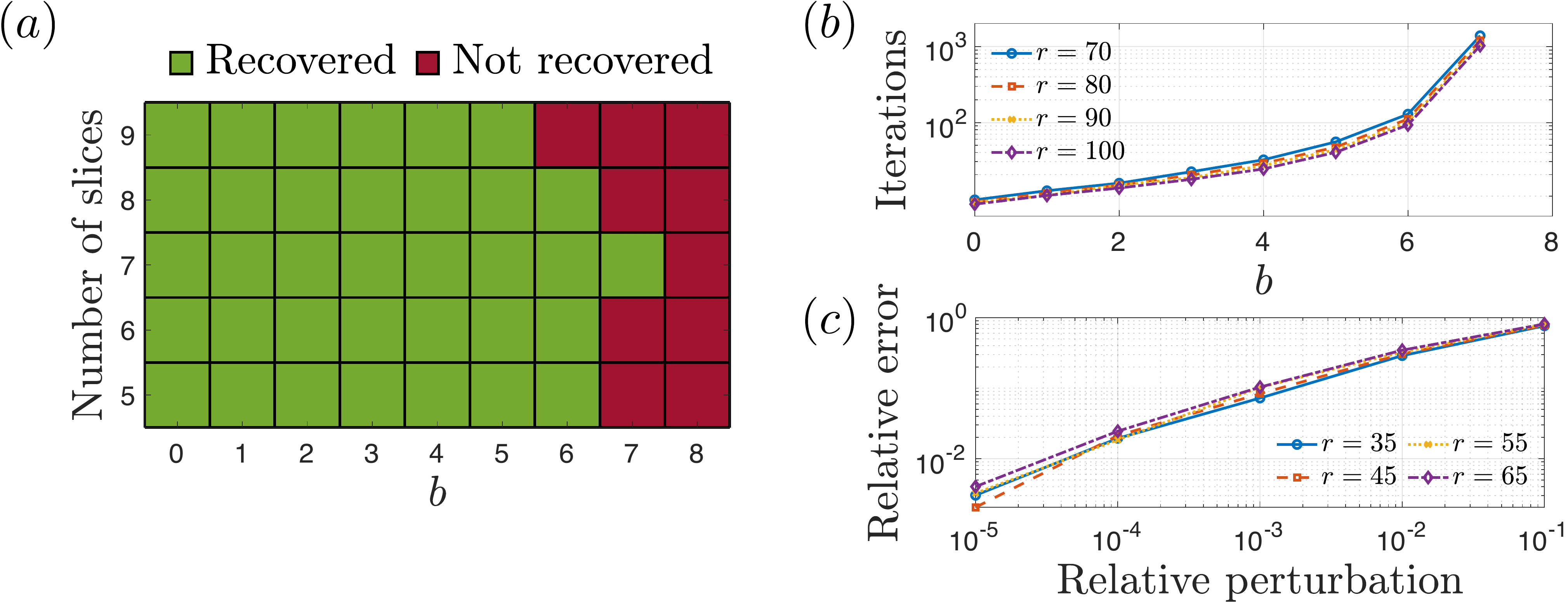}
\caption{Numerical results for the experiments in sections~\ref{sec:numerical1}, \ref{sec:numerical2} and \ref{sec:numerical3}. respectively. (a) Successfully recovered tensors as function of $b$ and the number of slices used. Green color indicates recovery and red non-recovery. (b) Number of iterations required for convergence of \eqref{eq:update} as a function of $b$ and $r$. (c) Effect of entrywise noise on recovery error, as function of the relative perturbation magnitude $\frac{\|E_{\text{noise}}\|}{\|T\|}$ and $r$.}
\label{fig:numerical}
\end{figure}

\subsection{Effect of number of slices on highest possible $b$}\label{sec:numerical1}
For a fixed value $n=100  = r$, we show the effect of recovery by the value of $b$ and the number of slices used. We used $\varepsilon_{\text{tol}} = 10^{-6}$. The results are shown in figure~\ref{fig:numerical}(a). Adding additional slices increases the maximum possible recoverable $b$ only up to a certain point, since our algorithm requires the slices to have disjoint index sets $\mathcal{I}_{k_i}(b)$ for all slices $i = 1, \ldots , m$ which imposes the bound on $b$ in terms of $m$ and $n$ in \eqref{eq:cond_disjoint}. The use of $m=7$ slices is clearly optimal in this case, and the maximum value of recoverable $b$ coincides with that of the lower bound on $n$ in \eqref{eq:main_bounds}.

\subsection{Number of iterations required for convergence as function $b$}\label{sec:numerical2}
For a fixed value $n=100 $, we show the number of iterations needed for convergence of the alternating least squares procedure in \eqref{eq:update}. For each value of $b$, we run $10$ randomly generated trials, and report the average number of iterations required for convergence. We used $\varepsilon_{\text{tol}} = 10^{-7}$. The results are shown in figure~\ref{fig:numerical}(b). Note that the number of required iterations remains relatively low for $b$ below the maximum admissible $b$ for this value of $n$. However, the number of required iterations increases for $b$ approaching its maximally recoverable value, and for decreasing $r$.

\subsection{Effect of entrywise noise on recovery error}\label{sec:numerical3}
For a fixed value $n=65$, we study the effect of an additional entrywise noise term on the recovery error. In detail, we generate random tensors $T$ of the form $\sum_{k=1}^{r} a_k \otimes b_k \otimes c_k$ with $a_k,b_k$ orthogonal vectors for $k=1,\ldots , r$, and apply the algorithm of the article to the tensor $T+E + E_{\text{noise}}$, where $E$ is the structured noise tensor discussed in the algorithm with $b=5$ and $E_{\text{noise}}$ is a tensor with independent normal entries. The algorithm of the article is then run on $T+E + E_{\text{noise}}$ with resulting tensor $T_0$. For each $b$, we run $100$ randomly generated trials with $\varepsilon_{\text{tol}} = 10^{-12}$. We show the resulting average recovery error $\frac{\|T-{T_0}\|}{\|T\|}$ as a function of $\frac{\|E_{\text{noise}}\|}{\|T\|}$ and $ r$ in figure~\ref{fig:numerical}(c).

\section{Conclusion}
We presented an algorithm to recover orthogonally decomposable tensors corrupted by arbitrarily strong, but structured noise. The problem can be seen as a deterministic tensor completion problem, and we have shown how this can be solved provided the tensor dimension and rank are larger than an affine function of the corruption bandwidth. Notably, this enables recovery under a high-rank assumption on the tensor, as opposed to low-rank assumptions commonly required for completion problems. The techniques in the article are not limited to the specific pattern of unknown elements treated. Future work therefore includes studying the possibility of bridging the gap between high- and low-rank completion techniques to develop algorithms with weaker conditions on the tensor rank. Another avenue for future work lies in relaxing the orthogonality condition in assumption~\ref{ass:ON} to soft-orthogonality constraints encoded instead by incoherent tensor components.

\appendix

\section{Proof of lemma~\ref{lemma:count}}\label{sec:proof_lemma_count}
We start with the first part of lemma~\ref{lemma:count}. For each slice $X_i$, only the variables not contained in the rows with indices in $\mathcal{I}_{k_i}(2b)$ enter into any of the equations in \eqref{eq:sylvester}. We therefore need to show that the linear operator in \eqref{eq:sylvester} has trivial kernel under the assumptions in \eqref{eq:cond_lemma}, for generic matrices $A,B,C$. This amounts to ensuring that the minors of the operator in \eqref{eq:sylvester} are non-zero. Since these minors are polynomial expressions in the elements of $A,B,C$, this will be true generically, provided that we can show that these polynomials are not identically vanishing. This will hold if we can exhibit one example of the operator in \eqref{eq:sylvester} with full column rank, under the bounds in \eqref{eq:cond_lemma}. We first prove an auxiliary result.

\begin{lemma}\label{lemma:lemma_proof}
Let $A_1, A_2, A_3$ be generic matrices in $\mathbb{R}^{m\times n}$. Let $I_1,I_2,I_3$ be disjoint subsets of $\llbracket 1, n\rrbracket$ and let $Y_1, Y_2, Y_3$ be matrices in $\mathbb{R}^{n\times p}$. If
\begin{enumerate}
\item $Y_i$ has non-zero elements only in the rows with indices in the set $I_i$, for $i = 1,2,3$
\item each column of $Y_i$ has at most $m$ non-zero elements, for $i = 1,2,3$
\end{enumerate}
then the system
\begin{align}\label{eq:lemma1_for_count1}
A_2Y_1 &= -(A_1Y_2)^T \\
A_3Y_1 &= -(A_1Y_3)^T \label{eq:lemma1_for_count2} \\
A_3Y_2 &= -(A_2Y_3)^T \label{eq:lemma1_for_count3}
\end{align}
has the unique solution $Y_1 = Y_2 = Y_3 = 0$.
\end{lemma}

\begin{proof}
Just as in the opening paragraph of this section, we need only find one example of matrices $A_1,A_2,A_3$ for which the conclusion holds, to ensure that it holds generically. To do this, let the columns of $A_1$ with indices in $I_2$ be zero. It follows that $A_1Y_2 = 0$, so, from \eqref{eq:lemma1_for_count1}, $A_2Y_1 = 0$. In the first column of this equation, discarding the zero elements of the first column of $Y_1$ gives a matrix equation for the non-zero elements in the first column of $Y_1$. From the second assumption in the lemma, the corresponding matrix has full column rank generically. It follows that the first column of $Y_1$ is zero, and similarly the remaining columns are as well. Inserting this into \eqref{eq:lemma1_for_count2} similarly gives $Y_3 = 0$, since $ I_2$ and $I_3$ are disjoint. Inserted into \eqref{eq:lemma1_for_count3}, this gives also $Y_2 = 0$, which concludes the proof.
\end{proof}

We now proceed to construct one example of the operator in \eqref{eq:sylvester} with full column rank. For any choice of three slices, we will show that the resulting variables $X_i$ are zero. which will suffice to prove the lemma. We first consider slices $i=1,3,5$. Let therefore $X_{1},X_{3},X_{5}$ be in the kernel of \eqref{eq:sylvester}. These matrices then satisfy

\begin{equation}\label{eq:sylvester_kernel}
X_{i}M_{j}^T - M_{j}X_{i}^T + M_{i}X_{j}^T - X_{j}M_{i}^T\Big|_{\mathcal{I}_{k_ik_j}(b)^\complement} = 0 , \quad i,j \in \{1,3,5\}.
\end{equation}

We will construct $M_1, M_3, M_5$ with support contained in disjoint rows. In detail, let $J_1, J_3, J_5$ be disjoint subsets of $\llbracket 1, n\rrbracket$ such that
\begin{align}
& J_i \cap \mathcal{I}_\ell(2b) = \emptyset, \text{ for } i,\ell \in \{1,3,5\}  ,  \label{eq:assumptions_for_proof0} \\
&\abs{J_i} = 4b+2 , \text{ for } i \in \{1,3,5\} ,  \label{eq:assumptions_for_proof1} \\
& \text{min}(J_3) \geq \text{max}(J_1) + 4b+2 , \label{eq:assumptions_for_proof2} \\
& \text{min}(J_5) \geq \text{max}(J_3) + 4b+2 .  \label{eq:assumptions_for_proof3}
\end{align}

Let $M_i$ have support in the rows contained in $J_i$. Note that this is possible when $r \geq \abs{J_1} + \abs{J_2} + \abs{J_3} = 3(4b+2)$ and $n \geq 8(4b+2)$, by letting the first $\abs{J_1}$ columns of $A$ have support contained in the rows with indices in $J_1$, the subsequent $\abs{J_3}$ columns of $A$ have support contained in the rows with indices in $J_3$ and the last $\abs{J_5}$ columns of $A$ have support in rows in $J_5$. If $D_{k_1}^c$ has support in the first $\abs{J_1}$ diagonal elements, $D_{k_3}^c$ in the subsequent $\abs{J_3}$ diagonal entries and $D_{k_5}^c$ in the subsequent $\abs{J_5}$ ones, the $M_i$ have the desired support.

We now show that this choice of $M_i$ enforces $X_1 = X_3 = X_5 = 0$. The matrix
\begin{equation}\label{eq:cols_support}
M_{i}X_{j}^T - M_{j}X_{i}^T,
\end{equation}
has support contained in the rows with indices in $J_i \cup J_j$. From \eqref{eq:sylvester_kernel}, it follows that the entries of $M_{i}X_{j}^T - M_{j}X_{i}^T$ not contained in $J_i \times J_i$, $J_i \times J_j$, $J_j \times J_i$, $J_j \times J_j$ or  $\mathcal{I}_{k_ik_j}(b)$ are zero. Therefore, elements with rows in $J_j$ and columns not contained in
\begin{equation}\label{eq:cols_not_zero}
\llbracket \text{min}(J_j)-2b, \text{max}(J_j)+2b\rrbracket \cup J_i \cup \mathcal{I}_{k_i}(2b) \cup \mathcal{I}_{k_j}(2b),
\end{equation}
are zero. We now claim that also the columns of $X_i^T$ not contained in \eqref{eq:cols_not_zero} are zero. To see this, observe that the non-zero elements $x$ of any such column appears in \eqref{eq:cols_support} as an equation of the form $Mx = 0$, where $M$ is a generic matrix with $\abs{J_j}$ rows. Since each column of $X_{i}^T$ has at most $4b+2$ non-zero elements, \eqref{eq:assumptions_for_proof1} enforces $x=0$.

Next, intersect the sets in \eqref{eq:cols_not_zero} for all $j \neq i$. Equations~\eqref{eq:assumptions_for_proof2}--\eqref{eq:assumptions_for_proof3} show that the columns of $X_i^T$ not contained in $J_i \cup \mathcal{I}_{k_i}(2b)$ are zero, for each $i=1,3,5$. We lastly show that the remaining columns of $X_i^T$ with indices in $J_i$ are zero as well. Equation~\eqref{eq:sylvester_kernel} now reads as
\begin{align}
M_3(J_3,:) X_1^T(:,J_1) &= -\Big(M_1(J_1,:)X_3^T(:,J_3)\Big)^T, \\
M_5(J_5,:)X_1^T(:,J_1) &= -\Big(M_1(J_1,:)X_5^T(:,J_5)\Big)^T,  \\
M_5(J_5,:)X_3^T(:,J_3) &= -\Big(M_3(J_3,:)X_5^T(:,J_5)\Big)^T.
\end{align}
By \eqref{eq:assumptions_for_proof0}--\eqref{eq:assumptions_for_proof3}, this system of equations satisfies the conditions in lemma~\ref{lemma:lemma_proof}. It follows that $X_1 = X_3 = X_5 = 0$ for generic matrices $A,B,C$ satisfying assumption~\ref{ass:ON}.

Next, for any remaining slice $X_{i}$, for $i \neq 1,3,5$, repeat the above argument with slices $1,3,i$ to conclude that also $X_i = 0$. This concludes the proof of the first part of lemma~\ref{lemma:count}.

We next prove the second part of lemma~\ref{lemma:count} and start with the bound on $n$ in \eqref{eq:cond_lemma_converse}. This comes from counting the number of equations present in \eqref{eq:sylvester} and enforcing that this equals at least the number of unknowns. Write
\begin{align}
n_i &= \text{dim}\left( \mathcal{U}(k_i)  \Big|_{\mathcal{I}_{k_i}(2b)^\complement\times \llbracket 1 , n\rrbracket} \right).
\end{align}

We first count the number of unknowns, i.e., $\sum_{i=1}^m n_i$. One can verify that
\begin{align}
n_1 &= (n-2(b+1))(3b+2) - \frac{b(b+1)}{2}, \\
n_2 &= (2b+1)(2n- 8b - 3) + b+1 - \frac{b(b+1)}{2}, \\
n_\ell &= (2b+1)(2n - 8b  - 2) - b(b+1), \quad \ell \not \in \{1,2,m\}.
\end{align}
For $n_m$, we write $\Delta = n - m\cdot (2b+1)$ and distinguish the two cases $\Delta \geq 0$ and $\Delta < 0$. We have
\begin{align}
n_m &= \begin{cases} (2b+1)(2n - 8b  - 2) - b(b+1), \quad \text{ if } \Delta \geq 0,\\
(n-\Delta_r)(2b+1) + \Delta_c(n-\Delta_r) - \frac{b(b+1)}{2}, \quad \text{ if } \Delta < 0,
\end{cases}
\end{align}
where
\begin{align}
\Delta_r &= n - (m-2)(2b+1)-2 ,\\
\Delta_c &= \begin{cases}
2b+1, \quad \text{ if } (m-1)(2b+1)+b+1 \leq n, \\
n - ((m-2)(2b+1)+b+1), \text{ if } (m-1)(2b+1)+b+1 \geq n.
\end{cases}
\end{align}

Next, we count the number of equations present in \eqref{eq:sylvester}. Figure~\ref{fig:num_eqns} illustrates the sparsity patterns of the equations retained in \eqref{eq:sylvester}.

\begin{figure}[!h]
\centering
\includegraphics[width=\textwidth]{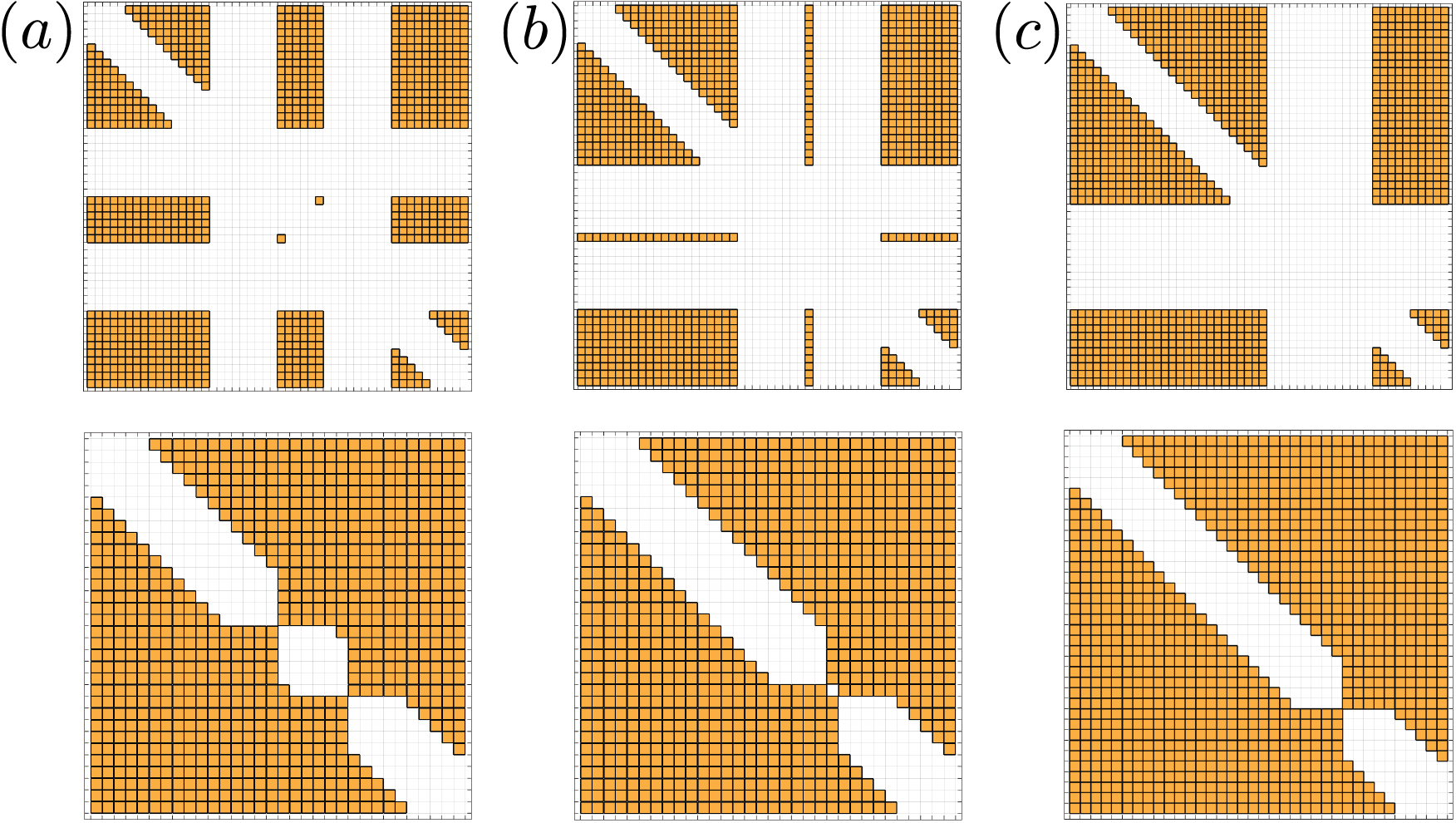}
\caption{Illustration of the equations retained in \eqref{eq:sylvester}. Orange color indicates that the equation in the corresponding element is retained, and white color that it is not retained. The top row shows the entire matrix and the bottom row removes the white space induced by the deleted rows and columns. $(a),(b)$ and $(c)$ illustrate $n_{58}$, $n_{68}$, and $n_{78}$, respectively.}
\label{fig:num_eqns}
\end{figure}

One can verify that the number of equations present in \eqref{eq:sylvester} for a specific choice of $i,j$ is
\begin{equation}
\frac{(\widehat{n} - 2b)(\widehat{n}-2b+1)}{2} + n_{ij},
\end{equation}
where $\widehat{n} = n - \abs{\mathcal{I}_{k_i}(2b) \cup \mathcal{I}_{k_j}(2b)}$ and the $n_{ij}$ are real numbers with $n_{ij} = n_{ji}$. With $i < j$, the $n_{ij}$ are determined by

\textbf{Case 1:} $\Delta > 2b$. If $i,j \not \in \{1,2\}$, then
\begin{equation}
n_{ij} =  \begin{cases}
b(2b+1) , \quad \abs{i-j} = 1, \\
b(2b+1) +2b, \quad \abs{i-j} = 2, \\
2b(2b+1), \quad \abs{i-j} \geq 3.
\end{cases}
\end{equation}

If $i = 1$, then
\begin{equation}
n_{1j} =  \begin{cases}
0 , \quad j = 2, \\
2b, \quad j = 3, \\
b(2b+1), \quad j \geq 4.
\end{cases}
\end{equation}

If $i = 2$, then
\begin{equation}
n_{2j} =  \begin{cases}
2b , \quad j = 3, \\
4b, \quad j = 4, \\
2b+b(2b+1), \quad j \geq 5.
\end{cases}
\end{equation}

\textbf{Case 2:} $\Delta \leq 2b$. Same as case 1, except subtract $\frac{(2b-\max(0,\Delta))(2b-\max(0,\Delta) + 1)}{2}$ from $n_{im}$.

\textbf{Case 3:} $\delta = n - (m-1)\cdot (2b+1) \leq 2b$. Same as case 2, except subtract $\frac{(2b-\delta)(2b-\delta + 1)}{2}$ from $n_{i,m-1}$.

Ensuring that the number of unknowns is at most the number of available equations with these expressions results in a lower bound for admissible $n$ in terms of $b$. This bound is shown in figure~\ref{fig:lower} for a few different values of $m$. The figure also shows the line $L(m)$ determined by $n = (2b+1)(m-1)+1$, which is the additional lower bound from \eqref{eq:cond_disjoint}, since the $m$ chosen slices $k_1, \ldots , k_m$ were required to be disjoint.

\begin{figure}[!h]
\centering
\includegraphics[width=\textwidth]{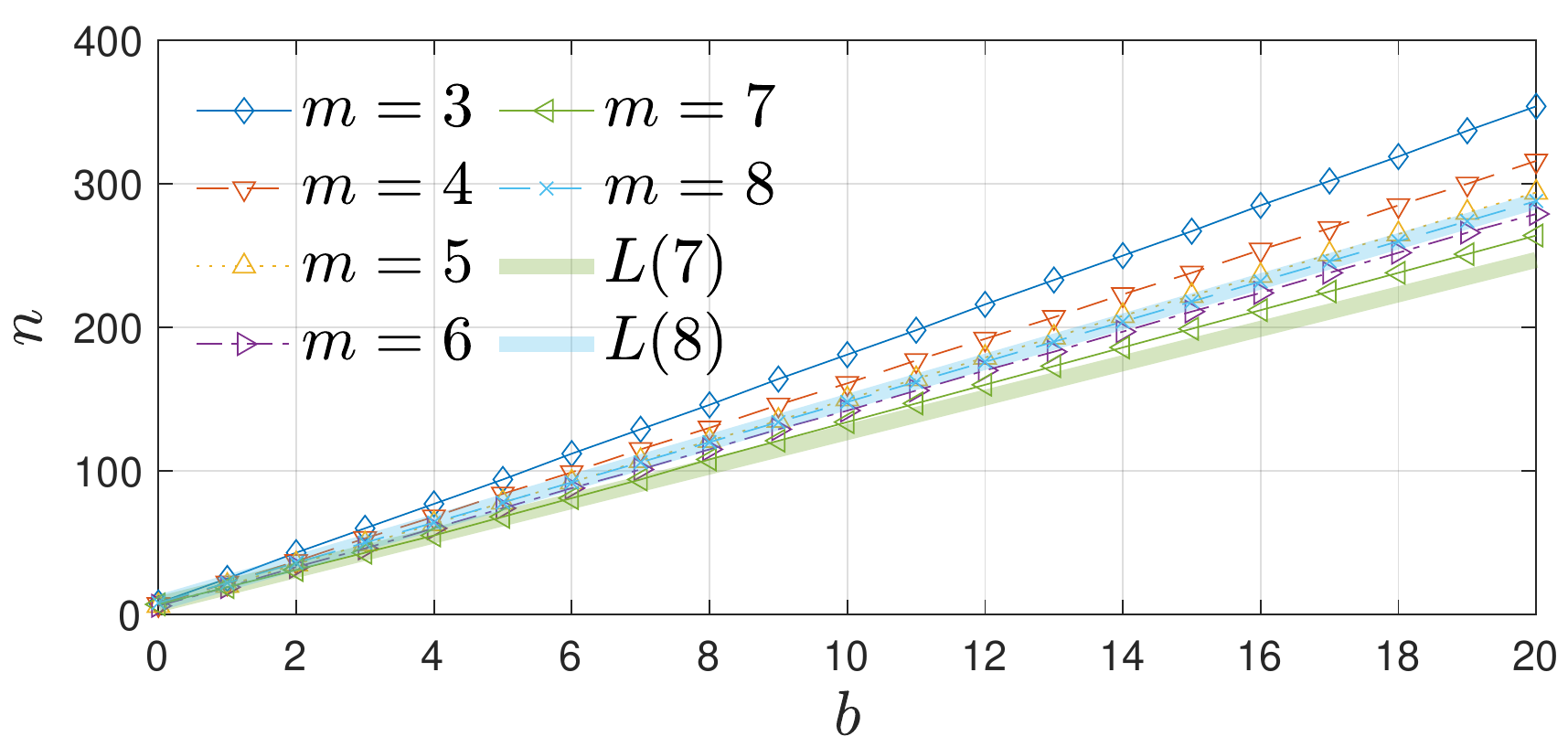}
\caption{Lower bounds of admissible $n$ as function of $b$.}
\label{fig:lower}
\end{figure}

Since the lower bound from using seven slices is contained above $L(7)$, but beneath $L(8)$, it follows that the optimal choice of $m$ is $m=7$. For this choice, explicitly writing out the bound results in
\begin{equation}
n \geq \text{max}\left(12b+7, \frac{93}{10}b + \frac{\sqrt{1419b^2 - 2704b + 921}}{10} - \frac{69}{10} \right),
\end{equation}
which is precisely the bound on $n$ in \eqref{eq:cond_lemma_converse}.

We lastly treat the bound on $r$ in \eqref{eq:cond_lemma_converse}. If $r < 4b + 2$, study a column of $X_2$ with $4b+2$ non-zero elements. For each row of $B^T$, pick out the columns with same sparsity pattern as the chosen column of $X_2$. The resulting matrix necessarily has linearly dependent columns. Take therefore $v$ in $\mathbb{R}^{4b+2}$ as a non-zero vector in the kernel of this matrix and define $X_2$ by inserting $v$ into the non-zero elements of the chosen column. Let all other entries of $X_2$ and $X_i$ for $i\neq 2$ be zero. Clearly, $X_{i} \in \mathcal{U}( k_i)$, for $i \in \llbracket 1, m\rrbracket$ and the $X_{i}$ are non-trivial elements of the kernel of \eqref{eq:sylvester}. This concludes the proof of lemma~\ref{lemma:count}.

\section{Proof of lemma~\ref{lemma:unique2nd}}
It is possible to prove lemma~\ref{lemma:unique2nd} in a similar fashion to the proof of lemma~\ref{lemma:count}. However, we present a different type of argument that allows for tighter bounds on $n$ and $r$ in terms of $b$. For ease of reference, we first prove the following two auxiliary results.
\begin{lemma}\label{lemma:sym_switch}
Let $M$ in $\mathbb{R}^{p\times q}$ have linearly independent columns and let $X$ be a matrix in $\mathbb{R}^{q\times p}$. The statements
\begin{enumerate}
\item $MX$ is a symmetric $p\times p$-matrix
\item $X^T = MS$, for $S$ a symmetric $q\times q$-matrix
\end{enumerate}
are then equivalent.
\end{lemma}
\begin{proof}
It is clear that the second statement implies the first. For the reverse direction, complete the column vectors of $M$ to a basis $m_1, \ldots , m_p$, and let $n_1, \ldots , n_p$ be a dual basis. Since $MX$ is symmetric, it can be written in the form
\begin{equation}
MX = \sum_{i,j=1}^{p} \alpha_{ij} m_i \otimes m_j,
\end{equation}
where $\alpha_{ij} = \alpha_{ji}$. For a fixed $\ell = q+1, \ldots , p$, acting on this equation with $n_{\ell}$ from the left yields
\begin{equation}
0 = n_{\ell}^T MX = \sum_{j=1}^{p} \alpha_{\ell j} m_j.
\end{equation}
By linear independence of the $m_j$, it follows that $\alpha_{\ell j} = 0 = \alpha_{j \ell}$, so we can write $MX = MSM^T$, where $S$ is the symmetric $q\times q$-matrix defined by $S_{ij} = \alpha_{ij}$. Acting on this equation by the pseudoinverse $M^{\dagger}$ from the left results in $X = SM^T$, so $X^T = MS$, which concludes the proof.
\end{proof}

\begin{lemma}\label{lemma:full_rank}
Let $I, J, K \subseteq \llbracket 1, n \rrbracket$ be index sets and $A,B,C \in \mathbb{R}^{n\times r}$ generic matrices satisfying assumption~\ref{ass:ON}. Let $i,j,k$ be distinct integers in $\llbracket 1, n\rrbracket$. Write
\begin{equation}
M_i = BD_i^c A^T, \quad M_{j} = BD_{j}^c A^T, \quad M_{k} = BD_{k}^c A^T.
\end{equation}
Then
\begin{enumerate}
\item The matrix
\begin{equation}  \label{eq:full_rank1}
\begin{bmatrix} M_{j}(I,J), & -M_{i}(I,K) \end{bmatrix}
\end{equation}
has full column rank if $J$ and $K$ are disjoint, $\abs{I} \geq \abs{J} + \abs{K}$, and $r \geq \abs{J} + \abs{K}$.
\item The matrix
\begin{equation} \label{eq:full_rank2}
\begin{bmatrix}
M_i(:,J), & -M_j(:,I) , &M_i(:,K) , & -M_k(:,I)
\end{bmatrix}
\end{equation}
has full column rank if $I,J$ and $K$ are disjoint, and $r \geq 2\abs{I} + \abs{J} + \abs{K}$.
\end{enumerate}
\end{lemma}
\begin{proof}
Just as in the first paragraph of appendix~\ref{sec:proof_lemma_count}, we need only produce examples of the matrices in statements $1$ and $2$ with full rank, under the given assumptions.

For statement $1$, choose $D_i^c$ to have non-zero diagonal elements for the indices contained in $J$, and $D_j^c$ in $K$, which is possible since $r \geq \abs{J} + \abs{K}$. Since $J$ and $K$ were assumed disjoint and $\abs{I} \geq \abs{J} + \abs{K}$, the entries of $A$ and $B$ with row indices in $I$ and column indices in $J$ and $K$ can be chosen separately, so that the matrix in \eqref{eq:full_rank1} equals e.g., the $\abs{I}\times(\abs{J}+\abs{K})$ identity matrix. This has full column rank, which concludes the proof of the first statement.

For statement $2$, let $D_i^c$, $D_j^c$, $D_k^c$ have non-zero elements only in the diagonal elements with indices in $I,J,K$, respectively. Since $I,J,K$ are disjoint, we can choose the columns of $A,B,C$ contained in $I,J,K$ separately to ensure the matrix in \eqref{eq:full_rank2} equals the $n\times (2\abs{I} + \abs{J} + \abs{K})$ identity matrix, with full column rank.
\end{proof}

We can now present the proof of lemma~\ref{lemma:unique2nd}.
\begin{proof}[Proof of lemma~\ref{lemma:unique2nd}]
We show that the kernel of \eqref{eq:solve_N2nd} is trivial, so take $X_{i}$ from this kernel, for $i\in I$. We first show that the rows of each $X_{i}$ contained in $\mathcal{I}_{k_i}(2b) \smallsetminus \mathcal{I}_{k_i}(b)$ are zero. To see this, we can express the fact that the $X_{i}$ are in the kernel of $\eqref{eq:solve_N2nd}$ by
\begin{equation}\label{eq:2ndfirst}
M_{i}^TX_{j} - M_{j}^T X_{i} = S_{ij}, \quad i,j \in I, i \neq j,
\end{equation}
where each $S_{ij}$ is a symmetric $n\times n$-matrix. Fix a pair $i,j$ in $I$ with $\mathcal{I}_{k_j}(3b) \cap \mathcal{I}_{k_i}(3b) = \emptyset$, which is possible under the assumptions on $I$ in \eqref{eq:defI1}--\eqref{eq:defI2}. We restrict \eqref{eq:2ndfirst} to rows and columns with indices not contained in $J := \mathcal{I}_{k_j}(3b)$. The columns of $X_i$ and $X_j$ not contained in $J$ have non-zero elements in rows with indices $\mathcal{I}_{k_i}(2b)$ and $\mathcal{I}_{k_j}(b)$, respectively. Taking this zero structure into account, \eqref{eq:2ndfirst} means that
\begin{equation}\label{eq:2ndproof}
\left[ M_{j}^{ T}(J^\complement,\mathcal{I}_{k_i}(2b)), -M_{i}^{T}(J^\complement,\mathcal{I}_{k_j}(b)) \right] \begin{bmatrix}
X_{i}(\mathcal{I}_{k_i}(2b),J^\complement) \\ X_{j}(\mathcal{I}_{k_j}(b),J^\complement)
\end{bmatrix},
\end{equation}
is a symmetric matrix. By lemma~\ref{lemma:full_rank}, the left matrix has full column rank provided $\abs{J^\complement} = n - (6b+1) \geq 6b+2$, i.e., $n \geq 12b+3$ and $r \geq 6b+2$. By lemma~\ref{lemma:sym_switch}, this means that 
\begin{equation}\label{eq:2ndclaim}
 \begin{bmatrix}
X_{i}(\mathcal{I}_{k_i}(2b),J^\complement) \\ X_{j}(\mathcal{I}_{k_j}(b),J^\complement)
\end{bmatrix}^T = \left[ M_{j}^{ T}(J^\complement,\mathcal{I}_{k_i}(2b)), -M_{i}^{ T}(J^\complement,\mathcal{I}_{k_j}(b)) \right] S,
\end{equation}
for some symmetric matrix $S$.

Next, the columns of $X_{i}(\mathcal{I}_{k_i}(2b),J^\complement)^T$ with indices in $\mathcal{I}_{k_i}(2b)\smallsetminus \mathcal{I}_{k_i}(b)$ have by construction at least $n - (4b+1) - \abs{J} = n - (10b+2)$ zero elements. The corresponding rows of
\begin{equation}
\left[ M_{j}^{ T}(J^\complement,\mathcal{I}_{k_i}(2b)), -M_{i}^{ T}(J^\complement,\mathcal{I}_{k_j}(b)) \right],
\end{equation}
are linearly independent by lemma~\ref{lemma:full_rank}, provided it has at least $n - (10b+2)$ columns, i.e., provided $n \geq 16b+4$, and $r \geq 6b+2$. It follows that the columns of $S$ with indices in $\mathcal{I}_{k_i}(2b) \smallsetminus \mathcal{I}_{k_i}(b)$ are zero. The same columns of $X_{i}^T$ are therefore zero as well. 

We next show that the remaining elements of the $X_{i}$ are zero. Equation~\eqref{eq:2ndfirst} now says that
\begin{equation}\label{eq:2ndthird}
\left[ M_{i}^{ T}(:,\mathcal{I}_{k_j}(b)), -M_{j}^{T}(\mathcal{I}_{k_i}(b)) \right] \begin{bmatrix}
X_{j}(\mathcal{I}_{k_j}(b), :) \\ X_{i}(\mathcal{I}_{k_i}(b),:)
\end{bmatrix}
\end{equation}
is a symmetric matrix. Lemma~\ref{lemma:sym_switch} shows that 
\begin{equation}\label{eq:2ndfourth}
\begin{bmatrix}
X_{j}(\mathcal{I}_{k_j}(b), :)^T, X_{i}(\mathcal{I}_{k_i}(b),:)^T
\end{bmatrix} = \left[ M_{i}^{ T}(:,\mathcal{I}_{k_j}(b)), -M_{j}^{ T}(:,\mathcal{I}_{k_i}(b)) \right] T_{ij},
\end{equation}
for $T_{ij}$ symmetric matrices. Every column of $X_{i}(\mathcal{I}_{k_i}(b),:)^T$ is therefore contained in the intersection of the ranges of the matrices $\left[ M_{i}^{T}(:,\mathcal{I}_{k_j}(b)), -M_{j}^{T}(:,\mathcal{I}_{k_i}(b)) \right]$, for $j$ in $I$, $j\neq i$. We now claim that this intersection consists of only the zero vector. Enumerating three distinct elements of $I$ as $i,j,\ell$, a vector in the intersection of these spaces is contained in the null space of the matrix
\begin{equation}
\begin{bmatrix}
M_i^{ T}(:,\mathcal{I}_{k_j}(b)), & -M_j^{ T}(:,\mathcal{I}_{k_i}(b)) , &M_i^{ T}(:,\mathcal{I}_{k_\ell}(b)) , & -M_\ell^{ T}(:,\mathcal{I}_{k_i}(b))
\end{bmatrix}.
\end{equation}

By lemma~\ref{lemma:full_rank}, this matrix has null space consisting of the zero vector if $r \geq 8b+4$. It follows that $X_{i} = 0$, which concludes the proof.

\end{proof}

\bibliographystyle{siamplain}
\bibliography{references}

\end{document}